\documentclass[11pt,reqno]{amsart}

\usepackage[T1]{fontenc}
\usepackage{lmodern}
\usepackage[letterpaper,margin=1in,heightrounded]{geometry}
\usepackage{amsmath,amssymb,mathtools}
\usepackage{booktabs,array}
\usepackage{microtype}
\usepackage{xurl}
\usepackage[hidelinks]{hyperref}

\hypersetup{
  pdftitle={An infinite family of doubly saturated R(3,t)-good graphs},
  pdfauthor={Abhishek Saigal and Akaash R. Parthasarathy},
  pdfsubject={Ramsey theory and saturated graphs},
  pdfkeywords={Ramsey graph, saturated graph, circulant graph,
    computer-assisted proof, Lean}
}

\allowdisplaybreaks[2]
\numberwithin{equation}{section}
\newtheorem{theorem}{Theorem}[section]
\newtheorem{proposition}[theorem]{Proposition}
\newtheorem{lemma}[theorem]{Lemma}
\newtheorem{corollary}[theorem]{Corollary}

\newcommand{\Z}{\mathbb Z}
\newcommand{\Zn}{\mathbb Z_n}
\newcommand{\Circ}{\operatorname{Circ}}
\newcommand{\DS}{\operatorname{DS}}
\newcommand{\norm}[1]{\lVert #1\rVert_n}
\newcommand{\iz}[2]{[#1,#2]_{\mathbb Z}}
\newcommand{\card}[1]{\lvert #1\rvert}

\title[An infinite family of doubly saturated Ramsey graphs]
{An infinite family of doubly saturated $R(3,t)$-good graphs}

\author[Abhishek Saigal]{Abhishek Saigal\textsuperscript{1,*}}

\author[Akaash R. Parthasarathy]{Akaash R. Parthasarathy\textsuperscript{2,*}}

\thanks{\textsuperscript{1}University of Illinois Urbana-Champaign.
\textsuperscript{2}Carnegie Mellon University.
\textsuperscript{*}Abhishek Saigal and Akaash R. Parthasarathy contributed
equally to this work.}


\begin{document}

\begin{abstract}
For every odd integer $t\ge17$, we prove that an explicit circulant graph
on $5t-10$ vertices is doubly saturated $R(3,t)$-good.  The graph is
triangle-free and has independence number $t-1$.  Adding any nonedge
creates a triangle, whereas deleting any edge creates an independent set
of order $t$.  This settles
Conjecture~2 of Przybocki, Mackey, Heule, and Subercaseaux.  A cyclic sumset
identity and explicit witnesses prove the local saturation properties.
Writing $t=2m+1$, a five-layer reduction proves the independence bound via
a uniform affine certificate for $m\ge30$ and an exhaustive checker for
$8\le m\le29$.  The checker soundness and the complete argument are
formalized in Lean~4.32.2.  Consequently,
$2t-1\le \DS(3,t)\le 5t-10$ for odd $t\ge17$.
\end{abstract}

\maketitle

\section{Introduction}

A graph is \emph{$R(s,t)$-good} if it contains neither a clique of order
$s$ nor an independent set of order $t$.  It is \emph{doubly saturated
$R(s,t)$-good} if it is $R(s,t)$-good, adding any missing edge destroys
that property, and deleting any existing edge also destroys it.  Let
$\DS(s,t)$ denote the minimum order of such a graph, when one exists.
This notion was introduced systematically by Przybocki, Mackey, Heule, and
Subercaseaux~\cite{PrzybockiEtAl2026}, building on earlier work on Ramsey
saturation~\cite{GrinsteadRoberts1982}.

For every positive integer $r$, write $\Z_r:=\Z/r\Z$.  For $n\ge3$ and
$S\subseteq\{1,\ldots,\lfloor n/2\rfloor\}$, write
$\Circ(n,S)$ for the graph with vertex set $\Zn$ in which $x$ and $y$
are adjacent exactly when their least circular distance belongs to $S$.
For integers $a\le b$, write
$\iz{a}{b}=\{a,a+1,\ldots,b\}$.
The following was Conjecture~2 of~\cite{PrzybockiEtAl2026}.

\begin{theorem}\label{thm:main}
Let $t\ge17$ be odd, put $n=5t-10$, and define
\[
 S_t=\{t-4,t-3\}
 \cup \iz{t+1}{(3t-9)/2}
 \cup \{(3t-5)/2,2t-4\}.
\]
Then $G_t=\Circ(n,S_t)$ is doubly saturated $R(3,t)$-good.  More
precisely,
\begin{enumerate}
\item $G_t$ is triangle-free.
\item $\alpha(G_t)=t-1$.
\item Adding any nonedge to $G_t$ creates a triangle.
\item Deleting any edge from $G_t$ creates an independent set of order $t$.
\end{enumerate}
\end{theorem}

Przybocki et al. computationally verified the construction for every odd
$t\le63$.  Our proof uses computer assistance, but the only
parameter-by-parameter exhaustive computation occurs for $8\le m\le29$.
All other finite checks are transported over affine parameter ranges.  An
additive identity in $\Zn$ proves both triangle-freeness and maximal
triangle-freeness.  The independence-number argument writes $n=5q$ and
views the graph as $q$ fibers with five layers.  After one doubled fiber is
normalized, a finite-state certificate forces at least as many empty fibers
as remaining doubled fibers.  Finally, deletion-criticality follows from
eight explicit families of independent-set witnesses.  Appendix
\ref{app:certificate} records the certificate data and Appendix
\ref{app:deletion} records all deletion witnesses.

The argument has also been formalized in Lean~4.32.2 using
Mathlib~\cite{deMouraUllrich2021,Mathlib2020}.  Section~\ref{sec:formal}
describes the formal statement, the executable certificates, and the
precise trust boundary.

\section{The cyclic sumset identity}\label{sec:sumset}

Write
\[
 t=2m+1,\qquad m\ge8,\qquad q=2m-1,
 \qquad n=10m-5=5q.
\]
The positive connection distances become
\begin{equation}\label{eq:S}
 S=\{2m-3,2m-2\}\cup\iz{2m+2}{3m-3}
     \cup\{3m-1,4m-2\}.
\end{equation}
For $x\in\Zn$, let $\norm{x}$ be its least circular distance from $0$, and
let
\[
 C=\{x\in\Zn:\norm{x}\in S\}=S\cup(-S)
\]
be the signed connection set.  Thus $x$ and $y$ are adjacent if and only
if $y-x\in C$.

\begin{lemma}[Sumset identity]\label{lem:sumset}
For every $m\ge8$,
\begin{equation}\label{eq:sumset}
 C+C=\Zn\setminus C.
\end{equation}
\end{lemma}

\begin{proof}
Put
\[
 P=\{2m-3,2m-2\},\quad I=\iz{2m+2}{3m-3},
 \quad c=3m-1,\quad s=4m-2.
\]
Every sum of two signed members of $C$ has circular distance either
$|a-b|$ or $\norm{a+b}$ for $a,b\in S$.  Direct interval subtraction gives
\[
\begin{array}{c|c@{\qquad}c|c}
\toprule
\text{pieces}&|a-b|&\text{pieces}&|a-b|\\
\midrule
P-P&\{0,1\}&I-I&\iz{0}{m-5}\\
I-P&\iz{4}{m}&c-P&\{m+1,m+2\}\\
c-I&\iz{2}{m-3}&s-P&\{2m,2m+1\}\\
s-I&\iz{m+1}{2m-4}&s-c&\{m-1\}\\
\bottomrule
\end{array}
\]
and hence
\begin{equation}\label{eq:diffunion}
 \{|a-b|:a,b\in S\}=\iz{0}{2m-4}\cup\{2m,2m+1\}.
\end{equation}
Circularizing the ordinary positive sums gives
\[
\begin{array}{c|c@{\qquad}c|c}
\toprule
\text{pieces}&\norm{a+b}&\text{pieces}&\norm{a+b}\\
\midrule
P+P&\iz{4m-6}{4m-4}&P+I&\iz{4m-1}{5m-5}\\
P+c&\{5m-4,5m-3\}&P+s&\{4m-1,4m\}\\
I+I&\iz{4m+1}{5m-3}&I+c&\iz{4m-1}{5m-6}\\
I+s&\iz{3m}{4m-5}&c+c&\{4m-3\}\\
c+s&\{3m-2\}&s+s&\{2m-1\}\\
\bottomrule
\end{array}
\]
and therefore
\begin{equation}\label{eq:sumunion}
 \{\norm{a+b}:a,b\in S\}
 =\{2m-1,3m-2\}\cup\iz{3m}{4m-3}
  \cup\iz{4m-1}{5m-3}.
\end{equation}
The union of~\eqref{eq:diffunion} and~\eqref{eq:sumunion} is precisely
$\iz{0}{5m-3}\setminus S$.  Since $n$ is odd, circular distance
parametrizes the inverse pairs of $\Zn$, proving~\eqref{eq:sumset}.
\end{proof}

\begin{proposition}\label{prop:triangle}
The graph $G_t$ is triangle-free, and adding any nonedge creates a
triangle.
\end{proposition}

\begin{proof}
A triangle would yield $a,b,c\in C$ with $a+b+c=0$, whence
$a+b=-c\in C$, contrary to Lemma~\ref{lem:sumset}.  Now let distinct
$x,y$ be nonadjacent.  Then $y-x\notin C$, so~\eqref{eq:sumset} gives
$y-x=a+b$ with $a,b\in C$.  The vertex $z=x+a=y-b$ is adjacent to both
$x$ and $y$.  Since $0\notin C$ and $a,b\in C$, this vertex is distinct
from both endpoints.  Thus adding $xy$ creates the triangle $xyz$.
\end{proof}

\section{The five-layer reduction}\label{sec:fibers}

Every $x\in\Z_n$ has a unique representation
\begin{equation}\label{eq:fibercoord}
 x=r+qi,\qquad 0\le r<q,\quad i\in\Z_5.
\end{equation}
For an independent set $A\subseteq\Z_n$, define its state in fiber $r$ by
\[
 B_r=\{i\in\Z_5:r+qi\in A\}.
\]

\begin{lemma}[Fiber model]\label{lem:fibermodel}
Each $B_r$ has at most two elements.  If $\card{B_r}=2$, then
$B_r=P_a:=\{a,a+1\}$ for some $a\in\Z_5$.  Moreover, for $0<d<q$ define
\begin{equation}\label{eq:gain}
 f(d)=
 \begin{cases}
 -1,&d=1,2,\\
  1,&3\le d\le m-2,\\
 -2,&d=m-1,\\
  1,&d=m,\\
 -2,&m+1\le d\le q-3,\\
  0,&d=q-2,q-1.
 \end{cases}
\end{equation}
Then, whenever $0\le r<s<q$,
\begin{equation}\label{eq:compat}
 B_s\cap\bigl(B_r+f(s-r)\bigr)=\varnothing.
\end{equation}
\end{lemma}

\begin{proof}
The connection $2q=4m-2\in S$ joins two layers in the same fiber exactly
when their difference is $\pm2$ modulo $5$.  A fiber therefore induces a
$5$-cycle, whose independent sets have size at most two and whose
two-element independent sets are the consecutive pairs $P_a$.

For the integer representative $1\le d\le q-1$, the following table gives
the unique signed connection above residue $d$ modulo $q$:
\begin{equation}\label{eq:signedlifts}
\begin{array}{c|c}
\toprule
d&d+qf(d)\in C\\
\midrule
1&-(2m-2)\\
2&-(2m-3)\\
3\le d\le m-2&d+q\in\iz{2m+2}{3m-3}\\
m-1&-(3m-1)\\
m&3m-1\\
m+1\le d\le q-3&d-2q\in\iz{-(3m-3)}{-(2m+2)}\\
q-2,\ q-1&2m-3,\ 2m-2\\
\bottomrule
\end{array}
\end{equation}
Indeed, these rows exhaust the $q-1$ inter-fiber signed connections, while
$\pm(4m-2)=\pm2q$ are the two connections above residue zero.  Thus there
is exactly one matching between two distinct fibers at residue difference
$d$, with layer displacement $f(d)$.  Avoiding that matching is exactly
\eqref{eq:compat}.
\end{proof}

Let $D$ and $E$ denote the numbers of double and empty fibers.  Every
other fiber contains one point, so
\begin{equation}\label{eq:count}
 \card A=q+D-E.
\end{equation}
Thus it is enough to prove $D\le E+1$.

If $D>0$, choose a double $B_r=P_a$ and translate $A$ by $-r-qa$.  This
graph automorphism normalizes that double to
\begin{equation}\label{eq:base}
 B_0=\{0,1\}.
\end{equation}
In every other fiber $r$, the base double excludes layers $f(r)$ and
$f(r)+1$.  Relabel the three surviving layers
\[
 f(r)+2,\qquad f(r)+3,\qquad f(r)+4
\]
as $0,1,2$.  Define $H_m$ on the vertex set
\[
 V(H_m)=\{(r,z):1\le r<q,\ z\in\{0,1,2\}\},
\]
where $(r,z)$ represents the surviving layer $f(r)+2+z$ in fiber $r$.
Write $z_r$ for $(r,z)$.  Adjacency in $H_m$ is inherited from $G_t$.
For $1\le r<s<q$, put
\begin{equation}\label{eq:h}
 h(r,s)=f(s-r)-f(s)+f(r)\pmod5.
\end{equation}

\begin{lemma}[Normalized adjacency]\label{lem:normalized}
For $m\ge8$ one has $h(r,s)\in\{1,2,4\}$.  Between fibers $r<s$, the
edges of $H_m$ are
\begin{equation}\label{eq:Hedges}
\begin{array}{c|c}
\toprule
h(r,s)&\text{edges}\\
\midrule
1&0_r1_s,\ 1_r2_s\\
4&1_r0_s,\ 2_r1_s\\
2&0_r2_s\\
\bottomrule
\end{array}
\end{equation}
and the only within-fiber edge is $0_r2_r$.
\end{lemma}

\begin{proof}
Substitute the six branches of~\eqref{eq:gain} into~\eqref{eq:h}.  The
values $0$ and $3$ are excluded by the corresponding interval
inequalities.  Substitution in~\eqref{eq:compat} then gives
Table~\eqref{eq:Hedges}.  The within-fiber assertion is the restriction of
the original $5$-cycle to the three surviving layers.  This finite symbolic
case analysis is verified uniformly in the formal development.
\end{proof}

The two possible doubles in $H_m$ are consequently
\[
 \text{type }0:\{0,1\},\qquad
 \text{type }1:\{1,2\}.
\]
Let $D_H$ and $E_H$ be the numbers of double and empty nonbase fibers.
Restoring the base fiber gives
\begin{equation}\label{eq:normalizedcount}
 D=1+D_H,\qquad E=E_H,\qquad
 \card A=q+1+D_H-E_H.
\end{equation}
The central claim is therefore the following.

\begin{proposition}[Normalized hole inequality]\label{prop:hole}
For every $m\ge8$, every independent configuration in $H_m$ satisfies
\begin{equation}\label{eq:hole}
 D_H\le E_H.
\end{equation}
\end{proposition}

Sections~\ref{sec:large} and~\ref{sec:small} prove this proposition in two
parameter ranges.

\section{The uniform affine certificate}\label{sec:large}

Assume throughout this section that $m\ge30$.  Two doubles of the same
type can coexist only if $h(r,s)=2$.  Exact interval arithmetic gives the
following stronger structural statement: every such pair has one endpoint
in
\begin{align*}
 L&=\{1,2,m-3,m-2,m,q-4,q-3\},\\
 R&=\{3,4,m-1,m+1,m+2,q-2,q-1\}
\end{align*}
and the other endpoint in the opposite set.  In particular, three doubles
of one type cannot be pairwise compatible.  Hence
\begin{equation}\label{eq:fourdouble}
 D_H\le4.
\end{equation}
The complete list of $21$ possible $h=2$ pairs is given in
Appendix~\ref{app:h2}.

Fix a typed double at a fiber $a$.  A set $K$ of other fiber indices is an
\emph{unsatisfiable core} if there is no selection of one vertex from every
fiber in $K$ whose union with the fixed double is independent in $H_m$.
Consequently, whenever an independent configuration contains that double,
at least one fiber indexed by $K$ is empty.  A pack of $p$ pairwise
disjoint cores therefore forces at least $p$ distinct empty fibers.

For example, fix the type-$0$ double at fiber $1$.  For
$K=\{m-2,m,m+2\}$, avoiding the fixed double leaves the respective layer
domains $\{1,2\}$, $\{0\}$, and $\{1,2\}$.  Choosing layer $0$ in the
middle fiber forces layers $2$ and $1$ in the two outer fibers, and those
two vertices are adjacent.  Hence $K$ is an unsatisfiable core.  All rows
in Appendix~\ref{app:cores} are checked in the same finite way.

The affine core table in Appendix~\ref{app:cores} proves that a double at
index $a$, of either type, forces at least $\ell(a)$ holes, where
\begin{equation}\label{eq:ell}
 \ell(a)=
 \begin{cases}
 1,&a\in\{1,2,q-2,q-1\},\\
 2,&a\in\{m-1,m\},\\
 3,&a\in\{m-2,m+1\},\\
 4,&\text{otherwise}.
 \end{cases}
\end{equation}
Each core has at most five fibers.  Once the affine order and gain branches
are fixed, its unsatisfiability is a truth table with at most $3^5$
assignments.  The certificate also verifies bounds, disjointness, and
coverage of every index for all $m\ge30$.  It does not instantiate a bounded
set of values of $m$.

Suppose for contradiction that $D_H=k>E_H$.  By~\eqref{eq:fourdouble},
$1\le k\le4$.  A double at $a$ would force $E_H\ge\ell(a)$, so every
double must satisfy $\ell(a)<k$.  The possible indices are therefore
\begin{equation}\label{eq:residualranges}
\begin{array}{c|c}
\toprule
k&\text{possible double indices}\\
\midrule
1&\varnothing\\
2&\{1,2,q-2,q-1\}\\
3&\{1,2,m-1,m,q-2,q-1\}\\
4&\{1,2,m-2,m-1,m,m+1,q-2,q-1\}.\\
\bottomrule
\end{array}
\end{equation}
For $k=2,3,4$, pairwise compatibility leaves respectively $12$, $14$,
and $1$ typed configurations.  The residual table in
Appendix~\ref{app:residual} gives $k$ distinct fibers whose layer domains
are empty in each case.  This contradicts $E_H<k$.  Hence
Proposition~\ref{prop:hole} holds for $m\ge30$.

\section{The finite parameter range}\label{sec:small}

It remains to prove Proposition~\ref{prop:hole} for $8\le m\le29$.  This
range is handled by an exhaustive Boolean checker whose soundness and
successful evaluation are both proved in Lean.

There are $F=2m-2$ normalized fibers.  For each type separately, the
checker enumerates all compatible triples of doubles and verifies that none
exists.  Hence $D_H\le4$.  For every $d\in\{1,2,3,4\}$ it then enumerates
every compatible typed configuration of $d$ doubled fibers.  It accepts
only if, for every such configuration, a recursive solver reports that no
completion exists using $F-2d+1$ distinct singleton fibers.  At each
remaining fiber, the recursion tries
each of the three layers and the option of leaving that fiber empty, so it
is exhaustive.

The soundness argument is short.  If $D_H=d$ and $E_H<d$, then there are
$F-d-E_H\ge F-2d+1$ singleton fibers, contradicting the negative completion
result.  In Lean, \path{completionSearch_eq_true_iff} proves that the
executable recursion is equivalent to its propositional specification, and
\path{smallChecker_sound} proves that a successful checker run implies
$D_H\le E_H$.  A compiled decision evaluates all $22$ values
$m=8,\ldots,29$.  This completes the proof of
Proposition~\ref{prop:hole}.

\begin{corollary}\label{cor:alpha}
For every $m\ge8$, $\alpha(G_t)=2m=t-1$.
\end{corollary}

\begin{proof}
If an independent set has no double fiber,~\eqref{eq:count} gives
$\card A\le q$.  Otherwise, Proposition~\ref{prop:hole} and
\eqref{eq:normalizedcount} give
$\card A\le q+1=2m$.  Conversely, $\card S=m$, and every $d\in S$ satisfies
$0<d<n/2$.  The connection residues $\pm d$ are therefore distinct, so
$G_t$ is $2\card S=2m$-regular.  By Proposition~\ref{prop:triangle}, the
neighborhood of any vertex is an independent set of order $2m$.
\end{proof}

\section{Deletion-criticality}\label{sec:deletion}

It remains to show that deleting any edge creates an independent set of
order $t$.  Translation and negation are automorphisms of $G_t$, so it is
enough to delete $\{0,d\}$ for $d\in S$.  Appendix~\ref{app:deletion}
defines, for each of the eight affine distance cases, a set
\[
 A_d\subseteq\Z_n,\qquad \card{A_d}=2m+1,\qquad 0,d\in A_d.
\]

\begin{proposition}[Deletion witnesses]\label{prop:deletion}
For every $m\ge8$ and $d\in S$, the only edge of $G_t[A_d]$ is
$\{0,d\}$.  Consequently, $A_d$ is independent after that edge is
deleted.
\end{proposition}

\begin{proof}
Represent the vertices by integers in $\iz{0}{n-1}$.  For $x<y$, the pair
$\{x,y\}$ is an edge exactly when $y-x$ belongs to
\begin{align}
 F={}&S\cup(n-S)\notag\\
  ={}&\{2m-3,2m-2\}\cup\iz{2m+2}{3m-3}
       \cup\{3m-1,4m-2\}\notag\\
    &{}\cup\{6m-3,7m-4\}\cup\iz{7m-2}{8m-7}
       \cup\{8m-3,8m-2\}.\label{eq:forbidden}
\end{align}
The complement of $F$ in $\iz{1}{n-1}$ is the disjoint union
\begin{equation}\label{eq:bins}
\begin{gathered}
\iz{1}{2m-4},\quad \iz{2m-1}{2m+1},\quad \{3m-2\},\\
\iz{3m}{4m-3},\quad \iz{4m-1}{6m-4},\quad
\iz{6m-2}{7m-5},\\
\{7m-3\},\quad \iz{8m-6}{8m-4},\quad
\iz{8m-1}{10m-6}.
\end{gathered}
\end{equation}
Each $A_d$ is a disjoint union of at most twelve affine integer intervals
and singletons.  If $[a_i,b_i]$ precedes $[a_j,b_j]$, all positive
differences between these blocks form exactly
\begin{equation}\label{eq:blockdiff}
 \iz{a_j-b_i}{b_j-a_i}.
\end{equation}
Likewise, the internal positive differences form $\iz{1}{b_i-a_i}$.
Applying
\eqref{eq:blockdiff} to the displayed witnesses yields $321$ affine
containments in the nine bins~\eqref{eq:bins}.  The sole exception is the
difference $d$ from $0$ to $d$.  The same affine certificate checks that
the blocks are ordered and disjoint and that their total cardinality is
$2m+1$.  All containments are uniform in $m\ge8$ (and in the auxiliary
parameter $r$ where present), and are machine-checked in Lean.  Thus
$\{0,d\}$ is the unique edge of $A_d$.
\end{proof}

\section{Proof of the theorem and the order bound}

\begin{proof}[Proof of Theorem~\ref{thm:main}]
Proposition~\ref{prop:triangle} gives triangle-freeness and shows that
adding any nonedge creates a triangle.  Corollary~\ref{cor:alpha} gives
$\alpha(G_t)=t-1$, so $G_t$ is $R(3,t)$-good.  Finally,
Proposition~\ref{prop:deletion}, followed by translation and negation,
shows that deleting any edge creates an independent set of order $t$.
These are exactly the four asserted properties.
\end{proof}

Przybocki et al.~proved the general lower bound
$\DS(s,t)\ge2s+2t-7$~\cite{PrzybockiEtAl2026}.  Taking $s=3$ and using
Theorem~\ref{thm:main} gives the following consequence.

\begin{corollary}\label{cor:DS}
For every odd $t\ge17$,
\[
 2t-1\le\DS(3,t)\le5t-10.
\]
\end{corollary}

\section{Formal verification and reproducibility}\label{sec:formal}

The complete argument is formalized in Lean~4.32.2 with Mathlib~4.32.2.
The formalization uses $k=m-8$ and defines
\[
 T(k)=2k+17,\qquad Q(k)=2k+15,\qquad
 N(k)=10k+75=5T(k)-10.
\]
The graph \texttt{Ramsey.G} at parameter $k$ has vertex type
$\Z_{N(k)}$ and the connection set of Theorem~\ref{thm:main}.  The
shifted result is \path{Ramsey.G_isDoublySaturatedRST}.  The final
theorem \path{Ramsey.odd_t_doubly_saturated_ramsey} says that, for
every odd $t\ge17$, there is a $k$ such that $t=T(k)$ and $N(k)=5t-10$,
and such that
\texttt{IsDoublySaturatedRST (G k) 3 t} holds.

The source modules mirror the mathematical proof as follows.
\begin{center}
\footnotesize
\renewcommand{\arraystretch}{1.02}
\begin{tabular}{@{}>{\raggedright\arraybackslash}p{0.22\textwidth}
                    >{\raggedright\arraybackslash}p{0.72\textwidth}@{}}
\toprule
role&principal modules\\
\midrule
local graph facts&
  \texttt{SumsetConsequences.lean}\newline
  \texttt{DirectedConnection.lean}, \texttt{LocalProperties.lean}\\
deletion witnesses&
  \texttt{Deletion*.lean}, \texttt{GraphTransport.lean}\\
five-layer reduction&
  \texttt{FiberReduction.lean}, \texttt{IndependenceConsequences.lean}\newline
  \texttt{CertificateInterface.lean}, \texttt{FiberModelBridge.lean}\\
finite certificate&
  \texttt{SmallCertificate.lean}, \texttt{SmallCertificateTransport.lean}\\
affine certificate&
  \texttt{LargeCertificate.lean}, \texttt{LargeCertificateTransport.lean}\\
final assembly&\texttt{Main.lean}\\
\bottomrule
\end{tabular}
\end{center}
The transport modules prove that the large-parameter truth tables apply over
complete affine ranges rather than sampled values.  The final module selects
the two regimes and invokes their common assembly theorem.

The formal proof sources underlying this manuscript are available at
\url{https://github.com/prog529007/ramsey} at commit \texttt{2381cb840c66}.
A clean verification uses
\begin{verbatim}
git clone https://github.com/prog529007/ramsey.git
cd ramsey
git checkout 2381cb840c660f998d9dac794edfe193088e6818
cd lean
lake build
\end{verbatim}
The Lake manifest pins Mathlib~4.32.2, at commit \texttt{905b95818eb3}.
From the repository root, \texttt{python3 verify\_all.py} runs separate Python
and C++17 checkers.  These provide redundant cross-checks and are not
dependencies of the Lean theorem.

For clarity about the trust boundary, the project source under
\texttt{lean/Ramsey} contains no \texttt{sorry}, \texttt{admit}, unsafe
declaration, or user-written axiom.  It uses \texttt{native\_decide} throughout
finite-state portions of the development, including the $22$-parameter
checker, finite layer classifications and model bridges, and truth tables
in the uniform affine certificate.  This is therefore a machine-checked
proof with compiled decision procedures, not a kernel-only proof.  Besides
Lean's kernel and the classical principles used by Mathlib, the trusted
base for these steps includes the \texttt{native\_decide} elaborator and the
native compilation path: Lean's compiler, code generator, runtime, and the
external compiler, linker, and platform.  Under the pinned toolchain, an
axiom audit of the final theorem reports \texttt{propext},
\texttt{Classical.choice}, \texttt{Quot.sound}, and $607$ generated
\texttt{\_native.native\_decide.ax\_*} declarations, but no
\texttt{sorryAx}.

\appendix

\section{The normalized certificate data}\label{app:certificate}

This appendix records the finite data used in the uniform proof of the
normalized hole inequality.  Throughout, $q=2m-1$ and $m\ge30$.

\subsection{Pairs with gain \texorpdfstring{$h=2$}{h=2}}\label{app:h2}

The pairs $1\le r<s<q$ for which $h(r,s)=2$ are exactly
\begin{align*}
&(1,3),(1,m-1),(1,m+1),(1,q-2),\\
&(2,3),(2,4),(2,m-1),(2,m+2),(2,q-2),(2,q-1),\\
&(m-3,m-1),(m-3,q-2),\\
&(m-2,m-1),(m-2,q-1),\\
&(m,m+1),(m,m+2),(m,q-2),(m,q-1),\\
&(q-4,q-2),(q-3,q-2),(q-3,q-1).
\end{align*}
Every pair has one endpoint in $L$ and one in $R$, for the sets $L,R$
defined in Section~\ref{sec:large}.  Thus this graph is bipartite, which is
the structural fact needed in~\eqref{eq:fourdouble}.

\subsection{Disjoint core packs}\label{app:cores}

A row
\[
 a:K_1\mid\cdots\mid K_p
\]
means that, after fixing a type-$0$ double at fiber $a$, the sets
$K_1,\ldots,K_p$ are pairwise-disjoint unsatisfiable cores.  Type-$1$
rows follow from the normalized-graph automorphism
\begin{equation}\label{eq:reflection}
 (r,z)\longmapsto(q-r,2-z),\qquad
 f(q-r)\equiv-1-f(r)\pmod5.
\end{equation}
In particular, a type-$0$ double at $a$ maps to a type-$1$ double at
$q-a$, and a core $K$ maps to $q-K$.
The following exceptional one-core rows, together with their reflections,
handle the end indices:
\begin{center}
\begin{tabular}{@{}cc@{}}
\toprule
fixed double&core\\
\midrule
$1_0$&$\{m-2,m,m+2\}$\\
$1_1$&$\{m-2,m,m+2\}$\\
$2_0$&$\{1,m,m+1,q-3\}$\\
$2_1$&$\{1,m,m+1\}$\\
\bottomrule
\end{tabular}
\end{center}
The middle packs are
\begin{align*}
m-2:\;&\{1,3,m\}\mid\{m-5,m-3,q-2\}
       \mid\{2,4,m-4,q-4\},\\
m-1:\;&\{m,m+1,m+3\}\mid\{q-3,q-2,q-1\},\\
m:\;&\{1,2,4\}\mid\{m-3,m-2,m-1,m+6\},\\
m+1:\;&\{m+3,m+4,q-3,q-2\}\mid\{2,3,m+2\}
       \mid\{m-2,m-1,q-1\}.
\end{align*}

For
\[
 a\in\iz{10}{m-8}\cup\{m+5,m+6\}\cup\iz{m+8}{q-7},
\]
the generic four-core row is
\begin{align*}
 K_1&=\{1,3,a-4,a-2\},&
 K_2&=\{2,4,a-3,a-1\},\\
 K_3&=\{a+1,a+3,m-3,m-1\},&
 K_4&=\{a+2,a+4,m-2,q-1\}.
\end{align*}
The remaining four-core rows are listed below.  The four math groups in
the right column are the four cores. Line breaks have no mathematical
significance.

\begingroup
\small
\setlength{\tabcolsep}{4pt}
\begin{center}
\begin{tabular}{@{}c p{0.87\textwidth}@{}}
\toprule
$a$&pairwise-disjoint cores\\
\midrule
$3$ & $\{4,m-3,m+4,q-2\}$ $\mid$
       $\{5,6,q-3,q-1\}$ $\mid$
       $\{m-2,m-1,m+3,m+5\}$ $\mid$
       $\{m,m+1,m+2\}$\\
$4$ & $\{1,m+1,m+3\}$ $\mid$
       $\{6,7,m-3,m-1\}$ $\mid$
       $\{3,m,m+2\}$ $\mid$
       $\{m+4,m+5,q-4,q-2\}$\\
$5$ & $\{3,m,m+2\}$ $\mid$
       $\{7,m+4,q-4,q-2\}$ $\mid$
       $\{6,8,m-3,m-1\}$ $\mid$
       $\{1,4,m+1,m+5\}$\\
$6$ & $\{7,m-2,m-1,m+7\}$ $\mid$
       $\{4,m,m+2,m+6\}$ $\mid$
       $\{3,5,q-4,q-2\}$ $\mid$
       $\{1,m+1,m+4,m+5\}$\\
$7$ & $\{2,m+7,m+9,q-2\}$ $\mid$
       $\{1,3,5\}$ $\mid$
       $\{8,m-2,m-1,m+6\}$ $\mid$
       $\{9,10,m,m+2\}$\\
\bottomrule
\end{tabular}
\end{center}

\begin{center}
\begin{tabular}{@{}c p{0.87\textwidth}@{}}
\toprule
$a$&pairwise-disjoint cores\\
\midrule
$8$ & $\{1,3,10,12\}$ $\mid$
       $\{2,m+2,m+8,m+9\}$ $\mid$
       $\{9,m+7,q-4,q-2\}$ $\mid$
       $\{5,6,m-2,m-1\}$\\
$9$ & $\{2,6,8,q-2\}$ $\mid$
       $\{m-2,m-1,m+6,m+8\}$ $\mid$
       $\{m,m+2,m+9,m+11\}$ $\mid$
       $\{1,3,7,10\}$\\
$m-7$ & $\{m,m+2,q-8,q-7\}$ $\mid$
       $\{1,3,m-8,m-5\}$ $\mid$
       $\{m-3,m-1,q-6,q-5\}$ $\mid$
       $\{m-11,m-9,q-3,q-1\}$\\
$m-6$ & $\{2,4,q-5,q-4\}$ $\mid$
       $\{q-7,q-6,q-3,q-1\}$ $\mid$
       $\{1,3,m-8,m-4\}$ $\mid$
       $\{m-5,m-3,m-1\}$\\
$m-5$ & $\{2,3,m-8,m-6\}$ $\mid$
       $\{m-3,m-2,m+14,q-3,q-2\}$ $\mid$
       $\{m-7,m,m+1,q-4\}$ $\mid$
       $\{1,m-4,m-1,q-5\}$\\
$m-4$ & $\{m-3,q-16,q-4,q-2\}$ $\mid$
       $\{2,3,m-6,q-6\}$ $\mid$
       $\{1,m-7,m-5,m+1\}$ $\mid$
       $\{m-13,m-2,m,q-1\}$\\
$m-3$ & $\{q-4,q-3,q-1\}$ $\mid$
       $\{1,3,m-7,m-5\}$ $\mid$
       $\{2,m-6,m-4,m+2\}$ $\mid$
       $\{m-2,m,m+1\}$\\
$m+2$ & $\{m+3,m+5,q-4,q-2\}$ $\mid$
       $\{m-2,m-1,m+1\}$ $\mid$
       $\{1,3,5\}$ $\mid$
       $\{m+4,m+6,q-3,q-1\}$\\
$m+3$ & $\{2,4,6\}$ $\mid$
       $\{m,m+5,m+7,q-2\}$ $\mid$
       $\{m-2,m-1,m+1\}$ $\mid$
       $\{3,m+4,q-3,q-1\}$\\
$m+4$ & $\{1,m-1,m+5,m+7\}$ $\mid$
       $\{m+2,m+6,q-3,q-2\}$ $\mid$
       $\{3,4,m-2,q-1\}$ $\mid$
       $\{m,m+1,m+3\}$\\
$m+7$ & $\{6,m-2,m-1,m+6\}$ $\mid$
       $\{2,7,m+9,q-1\}$ $\mid$
       $\{9,m-3,m+8,q-2\}$ $\mid$
       $\{8,m,m+1,m+5\}$\\
$q-6$ & $\{m-4,m-2,m-1,q-4\}$ $\mid$
       $\{2,m-6,m-5,m+2\}$ $\mid$
       $\{m-9,m-3,q-8,q-2\}$ $\mid$
       $\{q-5,q-3,q-1\}$\\
$q-5$ & $\{2,3,q-7,q-4\}$ $\mid$
       $\{1,m-7,m-6,q-2\}$ $\mid$
       $\{m-5,m-3,m-2,q-1\}$ $\mid$
       $\{m,m+2,q-6,q-3\}$\\
$q-4$ & $\{2,4,m-3,q-3\}$ $\mid$
       $\{1,m+1,q-8,q-6\}$ $\mid$
       $\{m-2,m-1,q-7,q-5\}$ $\mid$
       $\{m-6,m-5,m,m+2\}$\\
$q-3$ & $\{2,3,m-5,q-4\}$ $\mid$
       $\{1,m+1,q-6,q-5\}$ $\mid$
       $\{m-6,m-4,m,m+2\}$ $\mid$
       $\{m-3,m-2,m-1,q-8\}$\\
\bottomrule
\end{tabular}
\end{center}
\endgroup

The generic row, the boundary rows, and the middle and end exceptions cover
every index $1\le a<q$.  Every displayed affine expression lies in this
range for $m\ge30$. Each row has distinct, pairwise-disjoint cores.  The
formal certificate proves these assertions together with every core truth
table.  The reflection~\eqref{eq:reflection} then supplies the opposite
double type and proves the lower bound~\eqref{eq:ell} in all cases.

\subsection{Residual typed configurations}\label{app:residual}

We write $a_0$ and $a_1$ for doubles of types $0$ and $1$ at index $a$.
For each compatible configuration that survives the restrictions
in~\eqref{eq:residualranges}, the right column lists distinct fibers whose
available layer domains are already empty because of the displayed
doubles.

\begingroup
\small
\begin{center}
\begin{tabular}{@{}l l@{}}
\toprule
typed doubles&forced empty fibers\\
\midrule
$\{1_0,2_1\}$&$\{m,m+1\}$\\
$\{1_0,(q-2)_0\}$&$\{4,5\}$\\
$\{1_0,(q-1)_1\}$&$\{m-1,m\}$\\
$\{1_1,(q-2)_0\}$&$\{4,5\}$\\
$\{1_1,(q-2)_1\}$&$\{4,5\}$\\
$\{2_0,(q-2)_0\}$&$\{5,6\}$\\
$\{2_0,(q-1)_0\}$&$\{5,6\}$\\
$\{2_1,(q-2)_0\}$&$\{5,6\}$\\
$\{2_1,(q-2)_1\}$&$\{5,6\}$\\
$\{2_1,(q-1)_0\}$&$\{5,6\}$\\
$\{2_1,(q-1)_1\}$&$\{5,6\}$\\
$\{(q-2)_0,(q-1)_1\}$&$\{m-2,m-1\}$\\
$\{1_0,2_1,(m-1)_0\}$&$\{4,5,6\}$\\
$\{1_0,2_1,(q-2)_0\}$&$\{4,5,6\}$\\
$\{1_0,2_1,(q-1)_1\}$&$\{5,6,7\}$\\
$\{1_0,(q-2)_0,(q-1)_1\}$&$\{4,5,6\}$\\
$\{1_1,(m-1)_0,(q-2)_1\}$&$\{4,5,6\}$\\
$\{1_1,(m-1)_1,m_0\}$&$\{2,4,5\}$\\
$\{1_1,m_0,(q-2)_0\}$&$\{2,3,4\}$\\
$\{2_0,m_1,(q-2)_0\}$&$\{1,3,4\}$\\
$\{2_0,m_1,(q-1)_0\}$&$\{1,3,4\}$\\
$\{2_1,(m-1)_0,(q-2)_1\}$&$\{5,6,7\}$\\
$\{2_1,(m-1)_1,(q-1)_0\}$&$\{1,5,6\}$\\
\bottomrule
\end{tabular}
\end{center}

\begin{center}
\begin{tabular}{@{}l l@{}}
\toprule
typed doubles (continued)&forced empty fibers\\
\midrule
$\{2_1,(q-2)_0,(q-1)_1\}$&$\{5,6,7\}$\\
$\{(m-1)_1,m_0,(q-1)_0\}$&$\{1,2,3\}$\\
$\{m_1,(q-2)_0,(q-1)_1\}$&$\{1,3,4\}$\\
\midrule
$\{1_0,2_1,(q-2)_0,(q-1)_1\}$&$\{4,5,6,7\}$\\
\bottomrule
\end{tabular}
\end{center}
\endgroup

The three blocks contain respectively $12$, $14$, and $1$ configurations,
so the table is exhaustive for $D_H=2,3,4$.

\section{Deletion-witness families}\label{app:deletion}

All intervals below are inclusive integer intervals in
$\iz{0}{10m-6}$.  For every $m\ge8$, the pieces in each row are disjoint.

\begingroup
\setlength{\parskip}{0pt}
\setlength{\abovedisplayskip}{5pt plus 1pt minus 1pt}
\setlength{\abovedisplayshortskip}{5pt plus 1pt minus 1pt}
\setlength{\belowdisplayskip}{5pt plus 1pt minus 1pt}
\setlength{\belowdisplayshortskip}{5pt plus 1pt minus 1pt}
\newcommand{\witnesscase}[1]{%
  \par\addvspace{.3\baselineskip}%
  \noindent\emph{#1.}\par\nobreak}

\witnesscase{$d=2m-3$}
\begin{align*}
A_d={}&\{0,m-3,2m-3,2m-1,6m-4,7m-5,10m-7\}\\
&{}\cup\iz{m-1}{m}\cup\iz{5m-1}{6m-6}
\cup\iz{6m-2}{7m-7}.
\end{align*}

\witnesscase{$d=2m-2$}
\begin{align*}
A_d={}&\{0,m-1,2m-2,4m-1,5m-4\}\cup\iz{4}{m-3}\\
&{}\cup\iz{5m-2}{6m-7}\cup\iz{6m-2}{6m}
\cup\iz{10m-8}{10m-6}.
\end{align*}

\witnesscase{$d=2m+2$}
\begin{align*}
A_d={}&\{0,3,m-1,2m-1,2m+2,7m-3\}\cup\iz{m+1}{m+3}\\
&{}\cup\iz{5m+2}{6m-4}\cup\iz{6m-2}{6m-1}
\cup\iz{6m+1}{7m-5}.
\end{align*}

\witnesscase{$d=2m+2+r$, where $1\le r\le m-7$}
\begin{align*}
A_d={}&\{0,r+3,m-1,m+r+3,2m-1,2m+2+r,5m+r,7m-3\}\\
&{}\cup\iz{m+1}{m+r+1}\cup\iz{5m+r+2}{6m-4}\\
&{}\cup\iz{6m-2}{6m+r-1}\cup\iz{6m+r+1}{7m-5}.
\end{align*}

\witnesscase{$d=3m-4$}
\begin{align*}
A_d={}&\{0,3m-4,5m-5\}\cup\iz{4m-5}{4m-4}
\cup\iz{4m-1}{4m}\\
&{}\cup\iz{8m-5}{8m-4}\cup\iz{8m-1}{9m-8}
\cup\iz{9m-6}{10m-9}.
\end{align*}

\witnesscase{$d=3m-3$}
\begin{align*}
A_d={}&\{0,3m-3,5m-4,8m-4\}\cup\iz{4m-4}{4m-3}
\cup\iz{4m-1}{4m}\\
&{}\cup\iz{8m-1}{9m-7}\cup\iz{9m-5}{10m-8}.
\end{align*}

\witnesscase{$d=3m-1$}
\begin{align*}
A_d={}&\{0,m-1,3m-1\}\cup\iz{4m-1}{5m-6}
\cup\iz{5m-2}{5m-1}\\
&{}\cup\iz{9m-7}{9m-5}\cup\iz{9m-2}{10m-6}.
\end{align*}

\witnesscase{$d=4m-2$}
\begin{align*}
A_d={}&\{0,3,4m-2,8m-1\}\cup\iz{4m-1}{4m}\\
&{}\cup\iz{4m+2}{4m+3}\cup\iz{8m+2}{10m-6}.
\end{align*}
\endgroup

These cases partition $S$: the interval component is
\[
\iz{2m+2}{3m-3}
=\{2m+2\}\mathbin{\dot\cup}
 \{2m+2+r:1\le r\le m-7\}\mathbin{\dot\cup}
 \{3m-4,3m-3\}.
\]
For reference, the cardinality computations are
\[
\begin{array}{c|c}
\toprule
d&\card{A_d}\\
\midrule
2m-3&7+2+2(m-4)\\
2m-2&5+(m-6)+(m-4)+3+3\\
2m+2&6+3+(m-5)+2+(m-5)\\
2m+2+r&8+(r+1)+(m-r-5)+(r+2)+(m-r-5)\\
3m-4&3+2+2+2+(m-6)+(m-2)\\
3m-3&4+2+2+(m-5)+(m-2)\\
3m-1&3+(m-4)+2+3+(m-3)\\
4m-2&4+2+2+(2m-7)\\
\bottomrule
\end{array}
\]
and every entry equals $2m+1$.

\bibliographystyle{amsplain}
\bibliography{references}

\end{document}